\documentclass[12pt]{amsart}
\usepackage{amsmath,amssymb,amsbsy,amsfonts,amsthm,latexsym,
	amsopn,amstext,amsxtra,euscript,amscd,mathrsfs,graphics,epsfig,color}
\usepackage{graphicx}

\usepackage{float}

\newtheorem{theorem}{Theorem}
\newtheorem{remark}[theorem]{Remark}
\newtheorem{lemma}[theorem]{Lemma}

\newtheorem{conj}[theorem]{Conjecture}

\newtheorem{ques}[theorem]{Question}

\def\\{\cr}
\def\({\left(}
\def\){\right)}
\def\[{\left[}
\def\]{\right]}
\def\<{\langle}
\def\>{\rangle}
\def\fl#1{\left\lfloor#1\right\rfloor}

\def\li{\mathrm{li}\,}

\renewcommand{\vec}[1]{\mathbf{#1}}

\def\cA{\mathcal A}

\def\cC{\mathcal C}

\def\cG{\mathcal G}

\def\cL{\mathcal L}

\def\cP{\mathcal P}

\def\cK{{\mathcal K}}

\def\cU{\mathcal U}
\def\cV{\mathcal V}
\def\cW{\mathcal W}
\def\cX{\mathcal X}

\def\e{{\mathbf{\,e}}}

\def\gal{\text{\rm Gal}}

\def\F{\mathbb{F}}

\def\fp{\mathbb{F}_p}
\def\fq{\mathbb{F}_q}
\def\Z{\mathbb{Z}}
\def\K{\mathbb{K}}
\def\Q{\mathbb{Q}}
\def\L{\mathbb{L}}

\def\mand{\qquad \mbox{and} \qquad}

\def\notdivides{\mathrel{\kern-3pt\not\!\kern3.5pt\bigm|}}

\begin{document}

\title{On the Fixed Points of the Map $x \mapsto x^x$ Modulo a Prime}

\author{P\"ar Kurlberg}
\address{Department of Mathematics,
Royal Institute of Technology, SE-100 44 Stockholm, Sweden}
\email{kurlberg@math.kth.se}

\author{Florian~Luca}
\address{Fundaci\'on Marcos Moshinsky, Instituto de Ciencias
Nucleares,\linebreak UNAM, Circuito Exterior, C.U., Apdo. Postal 70-543, Mexico
D.F. 04510, Mexico}
\email{fluca@matmor.unam.mx}

\author{Igor Shparlinski}
\address{Department of Computing, Macquarie University, Sydney, NSW
  2109, Australia}
\email{igor.shparlinski@mq.edu.au}

\date{\today}

\maketitle

\begin{abstract}
  In this paper, we show that for almost all primes $p$ there is an
  integer solution $x\in [2,p-1]$ to the congruence $x^{x}\equiv
  x\pmod p$.  The solutions can be interpretated as fixed points of
  the map $x \mapsto x^x \pmod p$, and we study numerically and
  discuss some unexpected properties of the dynamical system
  associated with this map.  \end{abstract}

\section{Introduction}
\label{sec:introduction}

\subsection{Motivation}

For a prime $p$, we consider the properties of the map
$$
\psi_p:  \ x \mapsto x^x \pmod p
$$
when it acts on the integers $x \in [1,p-1]$. 
By  the results Crocker~\cite{Croc}
and Somer~\cite{Som}, there
are at least $\fl{\sqrt{(p-1)/2}}$ and at most
$3p/4 + O(p^{1/2 + o(1)})$, respectively,
distinct values of
$x^x \pmod p$ when $1\le x\le p-1$. 

We also note that various estimates depending on the multiplicative 
order modulo $p$ of $a$ on the number $T(p,a)$ of solutions of the congruence
\begin{equation}
\label{eq:fix val cong}
x^x \equiv a \pmod p,  \qquad  1 \le x \le p-1,
\end{equation}
have been given in~\cite{BBS1,BBS2}. In the most favorable case of $a=1$, 
by~\cite[Corollary~5]{BBS1}, we have
\begin{equation}
\label{eq:Tp1}
T(p,1) \le p^{1/3 + o(1)}
\end{equation}
as $p\to\infty$. 
Furthermore,  by~\cite[Bound~(2)]{BBS1}, for any integer $a$
we have $T(p,a) \le p^{11/12+o(1)}$.
Moreover, it is also shown in~\cite[Theorem~8]{BBS1} that the estimate
$$
\#\{1 \le x,y \le p-1~:~x^x \equiv y^y \pmod p\} \le  p^{48/25 + o(1)}
$$
holds as $p\to \infty$.

The map $\psi_p$ also appears  in some cryptographic protocols
(see~\cite[Sections~11.70 and~11.71]{MOV}), so it certainly
deserves more attention. 
Several  conjectures and numerical data concerning this map can be found
in~\cite{HoldMor}.  

Here, we address an apparently new problem and study the 
fixed points of the map $\psi_p$.  Let $F(p)$ denote the number of fixed points of the map $\psi_p$. 
That is, 
$$
F(p)   = \#\{ 1 \le x \le p-1~:~x^x \equiv x \pmod p\}. 
$$

Obviously $x=1$ is always a 
fixed points, which we call {\it trivial\/}. We show that for most 
primes $p$ the map $\psi_p$ has a
{\it nontrivial\/} fixed point $x \in [2,p-1]$. Thus, we 
are interested in primes $p$ with $F(p) > 1$.  
In the opposite direction, it has been noted in~\cite[Theorem~8]{BBS1} that the method used to prove 
\eqref{eq:Tp1} also applies to the congruence $x^{x-1} \equiv 1 \pmod p$, 
and thus it implies the bound
\begin{equation}
\label{eq:Fp}
F(p) \le p^{1/3 + o(1)}
\end{equation}
as $p\to\infty$.

We also study the quantity $F(p)$ and other  dynamical properties
(such as the period statistics) of the map 
$\psi_p$ numerically. In particular, these numerical results reveal that a na\"{\i}ve
point of view of treating $\psi_p$ as a ``random'' function on the set  $\{1, \ldots, p-1\}$
is totally wrong. In particular, the numerical results significantly deviate from those predicted for truly random maps by the work of Flajolet and  Odlyzko~\cite{FlOdl}. 
These results indicate that $\psi_p$ tends to have shorter orbits and more 
fixed points than a random map even after removing the trivial
fixed point $x=1$. On the other hand, it is highly likely that the bound~\eqref{eq:Fp} is very
far from being tight. We give some partial explanation for the ``non-randomness''
phenomenon, and introduce the notion of {\it random endomorphisms\/} in groups, which allows 
us to give some qualitative explanation for the numerical results.  We consider developing a rigorous analysis of the random endomorphisms to be a 
challenging and important open topic.

Finally, in Section~\ref{sec:comm ext}, we study the map $x \to
x^{f(x)} \pmod p$ for general polynomials $f(X) \in \Z[X]$, and show
that such a map can  have at most $p^{6/13+o(1)}$ fixed points, as
$p\to \infty$.

\subsection{Notation}

Before we give the  precise  statement we introduce some notation.

We define 
$\log x$ as $\log x=\max\{\ln x, 2\}$ where  $\ln x$ is the natural logarithm, 
Furthermore,  for an integer $k\ge 2$, we define
recursively $\log_k x=\log(\log_{k-1}x)$.

Throughout the paper, we use the Landau symbols $O$ and $o$ and the
Vinogradov symbols $\gg$ and $\ll$ with their usual meanings. We
recall that $A=O(B)$, $A\ll B$ and $B\gg A$ are all equivalent to
the fact that $|A|<cB$ holds with some constant $c$, while $A=o(B)$
means that $A/B\to 0$.

We further define the logarithmic integral 
$
\li(N) := \int_2^N \frac{d\, t}{\log t}.
$

We always use $p$ and $q$ for prime numbers.
We also use $\varphi(k)$ and $\omega(k)$ to denote the Euler function
and the number of distinct prime divisors of 
an integer $k$. 
 
Furthermore, $\F_p$ denotes a finite field of $p$ elements, which we
consider to be represented by the elements of the set $\{0,\ldots, p-1\}$,
while $\Z/q\Z$  denotes the residue ring modulo an integer $q \ge 1$.

\subsection{Heuristics on primes without nontrivial fix
  points}  

Let us write $\cA$ for the set of prime numbers $p$ for which $\psi_p$ does not have a nontrivial fixed point $x \in [2,p-1]$:
$$
\cA = \{p~\text{prime}~:~F(p) = 1\}. 
$$ 
One easily finds that  $\cA$ is not empty. 
In particular, among the first $1000$ primes, there are precisely $72$ of
them in $\cA$. The first few elements of $\cA$ are 
\begin{equation}
\label{eq:list}
3,~5,~7,~11,~53,~59,~83,~107,~179,~227,~269,~\ldots
\end{equation}
Quite likely, the set $\cA$ is infinite, but we have not been able to prove
this unconditionally. However, we can show this under some standard conjectures 
about prime numbers. For example, assume that 
\begin{equation}
\label{eq:special p}
p\equiv 3\pmod 8\mand p-1=2q,
\end{equation}
where $q$ is prime (several elements from the above list~\eqref{eq:list}: $11$, $59$, 
$83$, $107$, $179$, $227$,
are of this form). Consider an integer solution $x$ to $x^{x-1}\equiv
1\pmod p$. Then the multiplicative order of $x$ divides $x-1$, which
is an integer less than $p-1$. However, this multiplicative order
must also divide $p-1=2q$. So, the only possibilities are that the
order of $x$ is either $2$ or $q$. If it is $2$, then $x=1$ (which is 
excluded) or 
$x=p-1$, which is not a fixed point as $\psi_p(p-1) = 1$. If it is $q$, then $q\mid x-1$, and since
$x-1<2q$, we get that $x-1=q$, so $x=q+1=(p+1)/2$. Thus, we arrive at
$$
1\equiv x^{x-1}\equiv \({(p+1)}/{2}\)^{(p-1)/2}\equiv
2^{-(p-1)/2}\equiv 2^{(p-1)/2}\pmod p,
$$
by Fermat's Little Theorem, which, in particular, implies that $2$
is a quadratic residue modulo $p$. But this is impossible as
$p\equiv 3\pmod 8$. 

Standard conjectures then suggest that $\cA$ is
infinite, and, in fact, putting 
$$
\cA(N)=\cA\cap [1,N],
$$ 
the standard heuristic on the density of primes $p$ satisfying~\eqref{eq:special p}
makes us conjecture that 
the inequality 
$$
\#\cA(N)>c_0N/(\log N)^2
$$ 
holds for all $N\ge 2$
with some positive constant $c_0$.

In Section~\ref{sec:A heurist}, we  give some further heuristic arguments
suggesting that the stronger inequality
\begin{equation}
\label{eq:AN LB}
\#\cA(N)\ge \frac{N}{(\log N)^2} \exp\((1/\ln 2+o(1))\log_3 N \log_4 N\)
\end{equation}
holds as $N\to\infty$.  In fact, in Section~\ref{sec:F=1 heurist} we
also give a heuristic argument that the ``likelyhood'' of $\psi_p$
having no nontrivial fix points is of order $\exp(- \gamma(p) \cdot
\tau(p-1) )$, where $\tau(p-1)$ denotes the number of divisors of
$p-1$ and $\gamma(p) $ is some explicit but quite irregular function
of $p$ taking values in $ (0,1)$; see~\eqref{eq:Delta-p} for more
details.  In particular, we expect that $\psi_p$ is very likely to
have nontrivial fixed points unless the number of prime factors of
$p-1$ is very small.
\subsection{Main result} 

We obtain an unconditional result in the opposite direction of the
previous heuristics, in the sense that $\cA$ is fairly sparse. In
particular, the estimate $\#\cA(N)=o(\pi(N))$ holds as $N\to\infty$,
where, as usual, for a positive real number $x$ we use $\pi(x)$ to
denote the number of primes $p\le x$.

Let 
\begin{equation}
\label{eq:theta}
\vartheta =  \frac{1}{\zeta(2)} - \frac{1}{2\zeta(2)^{2}} 
=  \frac{6\pi^2 - 18}{\pi^4} \simeq 0.4231\ldots,
\end{equation}
where $\zeta(s)$ is the Riemann zeta-function.

\begin{theorem}
\label{thm:main}
We have
$$
\#\cA(N) \le \frac{\pi(N)}{(\log_3 N)^{\vartheta + o(1)}}
$$
as $N\to \infty$.
\end{theorem}

Our proof is based on an effective version of the Chebotarev Density Theorem
that is due to Lagarias and Odlyzko~\cite{LO}.

\section{Proof of Theorem~\ref{thm:main}}
\subsection{The strategy}

Observe that a nontrivial fixed point corresponds to a solution of the
congruence
\begin{equation}
\label{eq:Congr}
x^{x-1} \equiv 1 \pmod p, \qquad x \in \{2,3,\ldots,p-1\}.
\end{equation}
Thus, we wish to show that for almost all primes $p$ the
congruence~\eqref{eq:Congr} has a solution.

Given a
prime $p$ such that a ``small'' prime $q$ divides $p-1$, we write
$p-1 = q \cdot a$, so $a=(p-1)/q$.
For an integer $x$ of the form $x = 1+ a \beta$, with $\beta\in \{1,\ldots,q-1\}$, we have
$$
x^{x-1} \equiv ( 1 + a \beta )^{a \beta} \equiv (1 - \beta/q)^{\beta
(p-1)/q} \pmod p.
$$
Hence, we obtain a valid solution if $1 - \beta/q \equiv (q-\beta)/q
\pmod p$ is a $q$-th power modulo $p$ for some $0 < \beta <q$.  In
other words, with $n = q- \beta$, we find that 
$$
x=1+a\beta=1+\frac{1}{q}(p-1)(n-q)\in [2,p-1]
$$ 
is a solution to~\eqref{eq:Congr}
provided
that $n/q$ is a $q$-th power modulo $p$. Thus, it suffices 
to show that there exists a $q$-th power modulo $p$ of the form $n/q$ with $n \in [1,q-1]$.

Note that the
chance of a random element in the finite field of $p$ elements $\F_p$ being a $q$-th power
equals $1/q$. So, heuristically, assuming that the set of $q$-th powers 
has sufficiently random behavior, we can expect that the probability of 
this {\em not}  happening is 
$(1-1/q)^{q-1} = 1/e + o(1)$ as $q \to \infty$.

The strategy we adopt is thus to consider primes $p \equiv 1 \pmod
q$ for ``many'', say $k$, ``small'' (but not ``too small'') 
primes $q$; the ``probability'' that all
such $q$  fail to provide a valid solution $x$ to the original
congruence is expected to be about  $e^{-k}$, {\em provided} that we can show that
almost  all primes $p$ have such a property. 
We do this though not in a direct way. In particular, for the ``individual'' probability
of $q$  to fail we only obtain an upper  bound  of $1-\vartheta=0.576\ldots$ rather than
$1/e = 0.367\ldots$.

\subsection{The Chebotarev Density Theorem}

We let $\L$  be a  finite Galois extension of $\Q$ with
Galois group $G$ of degree $d= [\L:\Q]$ and discriminant
$\Delta$.  Let $\cC$ be a union of conjugacy classes of $G$.
We define
$$
\pi_\cC(N,\L/\Q) =\# \{ p \le N~:~p \text{ unramified in } \L/\Q, \
\sigma_p \in \cC\},
$$
where $\sigma_p$ is the Artin symbol of $p$ in the extension $\L/\Q$
(see~\cite{Gras}). 

A combination of  a version of the
Chebotarev Density Theorem due to Lagarias and Odlyzko~\cite{LO}
with a bound of Stark~\cite{Sta} for a possible Siegel zero, yields
the following result (see also~\cite[Lemma~6]{PomShp}).

\begin{lemma}
\label{lem:cheb} There are absolute constants $A_1, A_2 >0 $ such that
if
\begin{equation}
\label{eq:necessarybound}
\log N~\ge~10 d (\log |\Delta|)^2
\end{equation}
then
\begin{equation}
\begin{split}
\label{eq:Cheb}
\left|\pi_\cC(N,\L/\Q) -  \frac{\#\cC}{\#G} \li(N)\right| &\\
\ll
\frac{\#\cC}{\#G} &
\li\(N^\beta\) +
\|\cC\| N \exp \(-A_1 \sqrt{\frac{\log N}{d}}\)
\end{split}
\end{equation}
with some $\beta$ satisfying the inequality
$$
\beta~<~1 - \frac{A_2}{\max\{|\Delta|^{1/d}, \log |\Delta|\}},
$$
where $\|\cC\|$ is the number of conjugacy classes in $\cC$.
\end{lemma}

\subsection{Some preliminaries on Kummer extensions}
\label{sec:some-kumm-extens}

Let $q$ be prime.  We note that
$$
\{  p \leq N : p \equiv 1 \pmod q, \text{ $n/q$ is a $q$-th power
  modulo $p$} \}
$$
is, apart from the $O(\log(qn))$ ramified primes all dividing $qn$, equal to
the set of primes $p \leq N$ such that $p$ splits completely in the
Kummer extension $\K_{q,n} = \L_{q}(\sqrt[q]{n/q})$, where 
$\L_{q} = \Q(\zeta_{q})$ is
the cyclotomic extension generated by the primitive $q$-th root of
unity $\zeta_{q}=\e^{2\pi i/q}$. Note further that the condition that
$p$ splits completely in $\L_{q}$ is equivalent to $p \equiv 1 \pmod
q$.

The ideas behind our argument can be outlined as follows.
Note that  choosing a prime ideal $P\mid p$ in the ring of integers of
$\L_{q}$ essentially amounts to choosing a nontrivial $q$-th root
of unity in $\fp$.  Moreover, having made such a choice, the action
of the Artin map $\sigma_{P,n} \in \gal(\K_{q,n}/\L_{q})$ (note that
this Galois group is abelian) allows us, via Kummer theory, to
associate with an integer $n$ a canonical element in $\Z/q\Z$;
furthermore, this allows us to make ``compatible'' choices of
elements in $\Z/q\Z$ associated with different integers $n$.

To fix the ideas, let $g$ be a nontrivial $q$-th root modulo $p$. By
Kummer theory, we can then find ``compatible'' integers $x_{0},
x_{1},x_{2}, \ldots, x_{q-1}$ modulo $q$ such that $g^{x_{0}} \in q
\cdot (\fp^{\times})^{q}$, and $g^{x_{n}} \in n \cdot
(\fp^{\times})^{q}$ for $n=1,2,\ldots,q-1$ (where  $(\fp^{\times})^{q}$
is set of $q$-th powers in $\fp^{\times}$ and $ \lambda \cdot
(\fp^{\times})^{q}$ denotes the element-wise multiplication).

Note that knowledge of $x_{k}$ for all {\em prime} $k < q$,
determines $x_{n}$ modulo $q$ for $n$ {\em composite}.  Moreover,
the condition that $n/q$ is not a $q$-th power for all $n \in
[1,q-1]$ is equivalent to $x_{n} \not \equiv x_{0} \pmod q$ for $1
\leq n \leq q-1$.

\subsection{A system of linear forms modulo $q$}
\label{sec:syst-lin-form}

Motivated by the arguments of Section~\ref{sec:some-kumm-extens}, 
we  study a system of certain linear
equations modulo $q$. Let $d = \pi(q-1)$, and given an integer $n \in
[1,q-1]$, define a linear form $\cL_{n} : \fq^{d} \to \fq$ by
$$
\cL_{n}(\vec{v}) = \sum_{i=1}^{d} \alpha_{i,n} v_{i},
$$
where $\vec{v} = (v_1, \ldots, v_d)$ and the coefficients $\{\alpha_{i,n}\}$ are read from the prime factorization 
$$
n = \prod_{i=1}^d p_{i}^{ \alpha_{i,n}}.
$$

Given $x_{0} \in \fq$, we study
$$
N_{q} = \#\{ \vec{v} \in \fq^{d}~:~\cL_{n}(\vec{v}) \neq x_{0} \text{ for all
$ n \in \{1,2,\ldots,q-1\}$} \}.
$$

For $q$ large, it seems reasonable to expect that $N_{q}$ should be
of size $q^{d}/e$ since, for $\vec{v}$ a fixed nonzero vector, the
``probability'' that $\cL_n({\bf v}) \neq 
x_{0}$ for all $n$ {\em if the forms are randomly chosen}, equals
$(1-1/q)^{q-1} \simeq 1/e$. Equivalently, if we define
$$
c(q) = N_{q}/{q^{d}},
$$
we expect that $c(q) = 1/e + o(1)$ as $q\to \infty$. 

While we are not able to prove that $c(q)$ approaches $1/e$ as $q$
becomes large, we prove a weaker upper bound which is sufficient for
our purposes.

\begin{lemma}
\label{lem:cq}
As $q$ tends to infinity, we have
$$
c(q) \leq 1 - \vartheta + o(1),$$
where $\vartheta$ is given by~\eqref{eq:theta}. 
\end{lemma}

\begin{proof}
For $n>1$, the linear form $\cL_{n}$ is nontrivial and the
equation 
$\cL_{n}(\vec{v}) = x_{0}$ has at least one solution; hence
exactly $q^{d-1}$ solutions. Further, given two {\em square-free}
integers $2 \leq m<n < q$, we note that the corresponding linear
forms $\cL_{n}$ and $\cL_{m}$ are independent. Thus, there are exactly $q^{d-2}$
solutions $\vec{v}$ to
$$
\cL_{n}(\vec{v}) = \cL_{m}(\vec{v}) = x_{0}.
$$
Let $M$ denote the number of square-free
positive integers up to $q$. Thus, we have 
$M=(1/\zeta(2)+o(1))q$ as $q\to\infty$. 

To obtain an upper bound, we discard the condition that
$\cL_{n}(\vec{v}) \ne x_{0}$ for squarefull $n$. 
Then, removing those $\vec{v}$ for which $\cL_{n}(\vec{v}) = x_{0}$
for some square-free $n$, and adding back in $\vec{v}$'s for which 
$\cL_{n}(\vec{v}) = \cL_{m}(\vec{v}) = x_{0}$ for pairs of distinct
square-free $m,n$ (in essence, truncating the inclusion--exclusion 
principle at the third step), 
we find that
$$
N_{q} \leq q^{d} - M q^{d-1} + \binom{M}{2} q^{d-2}
=
q^{d}( 1- 1/\zeta(2) + 1/(2\zeta(2)^{2}) + o(1))
$$
as $q\to\infty$, and the result follows.
\end{proof}

\subsection{Independence of field extensions}
\label{sec:indep ext}

For a prime $q\mid Q$ we consider the algebraic number field 
$$
\K_{q} = \Q(\zeta_q,\sqrt[q]{2}, \sqrt[q]{3}, \sqrt[q]{5}, \ldots,
\sqrt[q]{q});
$$
that is, we adjoin the $q$-th roots of the unity and the $q$-th roots of the primes 
$p\le q$ to $\Q$.

Assume that $Q$ is a product of $k$ distinct primes $q_1,\ldots,q_k$. 
We define
$$
\K_{Q} = \K_{q_{1}} \circ \K_{q_{2}} \circ  \ldots \circ  \K_{q_{k}}
$$
to be the composite field obtained from the fields $\K_{q}$ as $q$ ranges over the
prime divisors of $Q$.

\begin{lemma}
\label{lem:lindisjoint}
Assume that $Q$ is an {\bf odd} integer. Then the field extensions
$\L_{q}(\sqrt[q]{\ell})/\L_{q}$ are linearly disjoint as $(q,\ell)$ ranges
over pairs of primes such that $\ell \leq q$ and $q|Q$.
\end{lemma}

\begin{proof}
We break the argument in two steps.

First we show that if $q$ is fixed, then $\L_q(\sqrt[q]{\ell})/\L_q$
are linearly disjoint once $\ell$ ranges over primes $\ell\le q$. If this
is  not so, then there exist $s\ge 2$ primes
$\ell_1,\ldots,\ell_s$ such that $\L_q\subsetneq \K$ where
$$
 \K=\L_q(\sqrt[q]{\ell_1},\ldots,\sqrt[q]{\ell_{s-1}})\cap
\L_q(\sqrt[q]{\ell_s}).
$$
Observe that $\K/\Q$ is normal as an intersection of normal
extensions. We show that $\K=\L_q(\sqrt[q]{\ell_s})$. Indeed, if this
is not so, then, by Galois theory, the group
$\gal(\L_q(\sqrt[q]{\ell_s})/\K)$ is a proper nontrivial normal
subgroup of $\gal(\L_q(\sqrt[q]{\ell_s})/\L_q)$, but this last group
has order $q$, a prime number. This shows that
$\K=\L_q(\sqrt[q]{\ell_s})$. So,
\begin{equation}
\label{eq:leftright} 
\L_q(\sqrt[q]{\ell_s})\subseteq
\L_q(\sqrt[q]{\ell_1},\ldots,\sqrt[q]{\ell_{s-1}}).
\end{equation}
The discriminant of the field on the left is divisible only by the
primes $q$ and $\ell_s$, while the discriminant of the field on the
right is divisible by the primes $q$ and $\ell_1,\ldots,\ell_{s-1}$. We get an
immediate contradiction unless $\ell_s=q$. So, it remains to treat the case $\ell_s=q$. If $s=2$, then we get
$$
\L_q(\sqrt[q]{q})\subseteq  \L_q(\sqrt[q]{\ell_1}).
$$
Since both extensions above have the same degree $q(q-1)$ over $\Q$, it follows that the above containment is in fact an equality. This is false because
$\ell_1$ ramifies in the field on the right but not in the field on the left.

Assume now that $s\ge 3$ is minimal such that containment~\eqref{eq:leftright} holds for some prime $q=\ell_s$ and some primes $\ell_1<\cdots<\ell_{s-1}<q$. 
Further, by the minimality of $s$, $\sqrt[q]{q}$ cannot belong to any field of the type $\Q(\zeta_q,\sqrt[q]\ell_{i_1},\ldots,\sqrt[q]{\ell_{i_t}})$ for some proper subset 
$\{i_1,\ldots,i_t\}$ of $\{1,\ldots,s-1\}$. Thus, we get a relation of the type
$$
\sqrt[q]{q}=R_0+R_1\sqrt[q]{\ell_{s-1}}+\cdots+R_{q-1} (\sqrt[q] {\ell_{s-1}})^{q-1},
$$
where $R_i=S_i(\zeta_q,\sqrt[q]{\ell_1},\ldots,\sqrt[q]{\ell_{s-2}})$ for some 
$$
S_i(X_0,X_1,\dots,X_{s-2})\in \Q[X_1,\ldots,X_{s-2}]
$$  
and at least one of 
$R_1,\ldots,R_{q-1}$ is nonzero. 
 Hence,  $\sqrt[q]{\ell_{s-1}}$ is an
algebraic number of degree at most $q-1$ over the normal field
$$\Q(\zeta_q,\sqrt[q]{q},\sqrt[q]{\ell_1},\ldots,\sqrt[q]{\ell_{s-2}}).$$
Since
$\Q(\ell_s^{1/q})$  is in fact of prime degree $q$ over $\Q$, we get that 
$$
\sqrt[q]{\ell_{s-1}}\in \Q(\zeta_q, \sqrt[q]{q},\sqrt[q]{\ell_1},\ldots,\sqrt[q]{\ell_{s-2}}),
$$ 
giving
$$
\Q(\sqrt[q]{\ell_{s-1}})\subseteq \Q(\zeta_q, \sqrt[q]{q},\sqrt[q]{\ell_1},\ldots,\sqrt[q]{\ell_{s-2}}).
$$
However, this last field inclusion is false because the discriminant of the field on the left is divisible by the prime $\ell_{s-1}$, while the discriminant 
of the field on the right is divisible only by primes $\ell_1,\ldots,\ell_{s-2}$ and $q$.

We next show that the fields $\K_q$ are linearly disjoint as $q$
varies over the prime factors of $Q$. Again assume that this is not
so and conclude that there exist $s\ge 2$ prime factors of $Q$
denoted $q_1<\cdots<q_s$ such that
$$ 
\Q\subset \K=\K_{q_1}\cdots \K_{q_{s-1}}\cap \K_{q_s}.
$$

Observe that all prime factors dividing the order of the  Galois group
of $\K_{q_s}/\Q$ divide  $q_s(q_s-1)$, while the Galois group of
$\K_{q_1}\cdots \K_{q_{s-1}}$ has order divisible only by primes
dividing $q_1(q_1-1)\cdots q_{s-1}(q_{s-1}-1)$. Thus, the order of
the Galois group $\gal(\K/\Q)$, as a factor group of
$\gal(\K_{q_s}/\Q)$, can be divisible only by primes dividing
$q_s-1$. 

The subgroup $\gal(\K_{q_s}/\K)$ is normal, so by the above
observation on possible prime divisors of its order, must contain the
$q_s$-Sylow subgroup of $\gal(\K_{q_s}/\Q)$, which is isomorphic to
$(\Z/q_s\Z)^{\pi(q_s)}$. However, the Galois group $\gal(\K_{q_s}/\Q)$
is isomorphic to a semidirect product of $\Z/{(q_s-1)\Z}$ with
$(\Z/q_s)^{\pi(q_s)}$, where the first cyclic group acts diagonally as
the group of automorphisms of $\Z/q_s\Z$.  It is not hard to see that
in the Galois group $\gal(\K_{q_s}/\Q)$, the $q_s$-Sylow subgroup is
maximal normal. This shows, via Galois correspondence between
subgroups and subfields, that $\gal(\K_{q_s}/\K)$ is the $q_s$-Sylow
subgroup, so $\K=\L_{q_s}$ is the cyclotomic field.

In particular, 
$\K$ contains $q_s$-th roots of unity and hence the discriminant
of $\K$ is divisible by $q_s$ --- a contradiction since the
discriminant of $\K_{q_1}\cdots \K_{q_{s-1}}$ is divisible only by
primes up to $q_{s-1}$. 

Altogether, this shows that the field extensions
$\L_{q}(\sqrt[q]{\ell})/\L_{q}$ are indeed linearly disjoint as
$(q,\ell)$ ranges over pairs of primes such that $\ell \leq q$,
thereby concluding the proof.
\end{proof}

\subsection{Estimating the degree and discriminant of $\K_Q$}
\label{sec:deganddisc}

We keep the notations from Section~\ref{sec:indep ext}. Put $d_Q$ and $\Delta_Q$ for the degree and discriminant of $\K_Q$, respectively.

\begin{lemma}
\label{lem:deganddisc}
The bounds
\begin{itemize}
\item[(i)] $d_Q\le \exp(q_k^2)$;
\item[(ii)] $\Delta_Q\le \exp(\exp(2q_k^2))$
\end{itemize}
hold for large enough $k$.
\end{lemma}

\begin{proof}
It is clear that $\K_Q$ is the compositum of 
\begin{equation}
\label{eq:boundforn}
n=(\pi(q_1)+1)+(\pi(q_2)+1)+\cdots+(\pi(q_k)+1)<\frac{q_k^2}{\log q_k}
\end{equation}
fields $\K_{i,j}=\Q(r_i^{1/q_j})$, where $r_i\in \{1\}\cup \{p\le q_j\}$ and $j=1,\ldots,k$, each of degree at most $q_k$. The inequality~\eqref{eq:boundforn} 
above holds for large $k$. Thus, (i) follows. For (ii), observe that 
the discriminant of each of $\K_{i,j}$ is at most $q_k^{2q_k}$. Label these fields in some way as $\K_1,\ldots,\K_n$ and let $\L_j=\K_1\circ \K_2\circ \cdots\circ \K_j$ for 
$j=1,\ldots,n$. Note that $\L_{j+1}=\L_j\circ \K_{j+1}$, therefore
$$
\Delta_{\L_{j+1}}\le \Delta_{\L_j}^{[\K_{j+1}:\Q]} \cdot \Delta_{\K_{j+1}}^{[\L_{j}:\Q]}.
$$
Since $[\K_{j+1}:\Q]\le q_k$, $[\L_j:\Q]\le q_k^j$ and $\Delta_{\K_j}\le q_k^{2q_k}$, we conclude that if we put $\lambda_j$ for some constant such that $\Delta_{\L_j}\le q_k^{\lambda_j q_k^j}$, then the inequalities
$$
\lambda_1\le 2\quad {\text{and}}\quad \lambda_{j+1}\le \lambda_j+2
$$
hold for $j=1,\ldots,n-1$. Hence, $\lambda_j\le 2j$ for $j=1,\ldots,n$. With $j=n$, we obtain
$$
\Delta_Q\le q_k^{2n q_k^{n}}<q_k^{2q_k^{q_k^2+2}}<\exp(\exp(2q_k^2))
$$
for all large $k$, thus proving (ii).  
\end{proof}

\subsection{Some technical estimates}
\label{sec:estimates}

For a square-free integer $S$, we define
$$
c(S) = \prod_{q\mid S} c(q).
$$
For  positive integers $L$ and $R$ with $Q = LR$, define
$$
\cP_{L,R}(N) = \#\{ p \leq N~:~\gcd(p-1,Q) = L \},
$$
and
$$
\widetilde{\cP}_{L,R} (N) = \{ p \in P_{L,R}(N)~:~n/q \not \in (\fp^{\times})^{q}
\text{ for all $q\mid L$ and $0<n<q$} \}.
$$
\begin{lemma} 
\label{lem:PP}
If 
\begin{equation}
\label{eq:ass}
6q_k^2<\log_2 N,
\end{equation}
then 
\begin{equation}
\label{eq:PI}
\widetilde{P}_{L,R}(N)   \ll    \pi(N) \cdot
\frac{c(L)}{\varphi(L)} \cdot \prod_{q\mid R}\left(\frac{q-2}{q-1}\right).
\end{equation}
wh
\end{lemma}

\begin{proof}
This follows from  the Chebotarev density theorem. 
 More precisely, a prime
$p$ counted by $\widetilde{P}_{L,R}(N)$ has the following
property: $p \equiv 1 \pmod q$ for each prime $q\mid L$   and for
all $1 \leq n < q$, $n/q$ is not a $q$-th power in $\fp^\times$.  In terms of the image of the Frobenius map, the relative size
of the corresponding conjugacy classes in $\gal(\K_q/\Q)$, is given by $c(q)$ (see Section~\ref{sec:syst-lin-form}).  Since by Lemma~\ref{lem:lindisjoint} the
field extensions $\K_{q_i}$ are linearly disjoint for $i=1,\ldots,k$, 
the relative size inside  $\gal(\K_L/\Q)$
is given by 
$c(L)$. This takes care of the main term. For the error term, 
we appeal to Lemmas~\ref{lem:cheb} and~\ref{lem:deganddisc}. 
More precisely, by Lemma~\ref{lem:deganddisc}, we have
$$
10d_Q (\log \Delta_Q)^2<10\exp(5q_k^2))<\log N
$$
for large $k$ by the assumption~\eqref{eq:ass},  so 
the inequality~\eqref{eq:necessarybound} holds. As for error terms, we
have 
$$
d_Q\le \exp(q_k^2)<(\log N)^{1/6}
$$
so the second error term in~\eqref{eq:Cheb} is negligible with respect to the main term. Finally, we note that the first error term in~\eqref{eq:Cheb}
is at most comparable with the main term and it could be incorporated into 
it given that~\eqref{eq:PI} is only an upper bound estimate.
\end{proof}

We now set
\begin{equation}
\label{eq:Qt}
Q_t = \prod_{t < q \leq e^{t}} q.  
\end{equation}
Thus, $Q$ has $k = \pi(e^t) - \pi(t)$ prime factors labeled $q_1,\ldots,q_k$. 
The inequality~\eqref{eq:ass} is satisfied for this choice of $Q$ provided that 
$N$ is large and 
\begin{equation}
\label{eq:choiceoft}
t=\frac{1}{3} \log_3 N.
\end{equation}
We get the following result.

\begin{lemma}
\label{lem:Sum PLRN}
If $N$ is large and~\eqref{eq:choiceoft} holds, then
$$
P_{1,Q_t}(N) \ll \frac{\pi(N)\log_ 4 N}{\log_3 N}. 
$$
\end{lemma}

\begin{proof} By the Brun sieve~\cite[Theorem~3, Section~I.4.2]{Ten}, and 
on recalling  Mertens formula~\cite[Section~I.1.5]{Ten}, we have
\begin{equation*}
\# \{ p \leq N~:~\gcd(p-1,Q_t)=1 \} \ll \pi(N)  \prod_{q\mid Q_t}\left(\frac{q-2}{q-1}\right)\ll \frac{\pi(N)\log t}{t},
\end{equation*}
and the result now follows from~\eqref{eq:choiceoft}. 
\end{proof}

\subsection{Concluding the proof}
\label{sec:concl}

We assume that $Q_t$ is  given by~\eqref{eq:Qt} 
where $t$ is given by~\eqref{eq:choiceoft}. In particular, the conditions of Lemmas~\ref{lem:PP}
and~\ref{lem:Sum PLRN} are satisfied. 

By Lemma~\ref{lem:PP}, we have
\begin{equation}
\label{eq:sum PLRN}
\sum_{LR = Q_t, L>1}
\widetilde{P}_{L,R}(N)
\ll \pi(N)
\sum_{LR = Q_t, L>1}
\frac{c(L)}{\varphi(L)} \cdot \prod_{q\mid R}\left(\frac{q-2}{q-1}\right).
\end{equation}
Furthermore, 
\begin{equation*}
\begin{split}
\sum_{\substack{LR = Q_t\\ L>1}}
\frac{c(L)}{\varphi(L)}  &\cdot \prod_{q\mid R}\(\frac{q-2}{q-1}  \)
= \prod_{q\mid Q_t}\left(\frac{q-2}{q-1}\right)  \sum_{\substack{LR = Q_t\\ L>1}}
\frac{c(L)}{\varphi(L)} \cdot\prod_{q\mid L}\frac{q-1}{q-2} \\
&= \prod_{q\mid Q_t}\left(\frac{q-2}{q-1} \right)\sum_{\substack{LR = Q_t\\ L>1}} c(L)  \cdot
\prod_{q\mid L}\frac{1}{q-2} \\
&\leq\prod_{q\mid Q_t}\left(\frac{q-2}{q-1} \right)
\sum_{L\mid Q_t}\prod_{q\mid L}\left(\frac{c(q)}{q-2}\right)\\
&=\prod_{q\mid Q_t}\left(\frac{q-2}{q-1} \right)
\prod_{q\mid Q_t}
\(1+\frac{c(q)}{q-2}\)=
\prod_{q\mid Q_t}
\(1-\frac{1-c(q)}{q-1}\).
\end{split}
\end{equation*}
Thus, recalling~\eqref{eq:sum PLRN}, we obtain
$$
\sum_{LR = Q_t, L>1}
\widetilde{P}_{L,R}(N)\ll 
\pi(N) \prod_{q\mid Q_t}
\(1-\frac{1-c(q)}{q-1}\).
$$

Using Lemma~\ref{lem:cq}  and then the Mertens formula again, we obtain
\begin{equation*}
\begin{split}
\prod_{q\mid Q_t}
\(1-\frac{1-c(q)}{q-1}\)& \ll  \exp \(- \sum_{q\mid Q_t} \frac{1-c(q)}{q}\) \\
& \ll   \exp \(- 
(\vartheta + o(1))\sum_{q\mid Q_t} \frac{1}{q}\)\\
& =   \exp \(- (\vartheta + o(1)) \log t\)=\frac{1}{(\log_3 N)^{\vartheta + o(1)}},
\end{split}
\end{equation*}
and so
$$
\sum_{LR = Q_t, L>1}
\widetilde{P}_{L,R}(N)\le 
\frac{\pi(N)}{(\log_3 N)^{\vartheta + o(1)}} 
$$
as $N\to\infty$.
With Lemma~\ref{lem:Sum PLRN}, we finally get that
\begin{equation*}
\begin{split}
\#\cA(N)& \ll P_{1,Q_t}(N)+\sum_{LR = Q_t, L>1}
\widetilde{P}_{L,R}(N)\\
& \ll \pi(N)\(\frac{\log_4 N}{\log_3 N}+\frac{1}{(\log_3 N)^{\vartheta + o(1)}}\),
\end{split}
\end{equation*}
as $N\to\infty$,  which finishes the proof.

\section{Further Remarks on  $\#\cA(N)$ }
\label{sec:A heurist}

\subsection{Heuristic arguments}
\label{sec:F=1 heurist}

Recall that $x=1$ is always a trivial fixed point, and note that
$x=p-1$ is never a fixed point.  Hence, we only  consider $x$
whose multiplicative order is greater than two, and the exponent
$x-1$ ranging over integers in the interval $[1,p-3]$.

If $d|p-1$ and $x$ is a primitive $d$-th root of unity {\em and} we
make the 
assumption that the exponent $x-1$ is ``independent'' of $x$, the
``chance'' that $x^{x-1} \equiv 1 \pmod p$ equals the chance that
$d|x-1$; this occurs with probability 
\begin{equation}
  \label{eq:d-power-probability}
\frac{\fl{(p-3)/d} }{p-3}= \frac{(p-1)/d-1}{p-3} = 1/d + O(1/p).
\end{equation}

Letting $x$ range over the set of $\varphi(d)$ primitive $d$-th roots of
unity, the probability that $x^{x-1} \not \equiv 1 \pmod p$ for all of
them, assuming independence, equals
$
\left(1- \frac{\fl{(p-3)/d} }{p-3} \right)^{\phi(d)}.
$
Moreover, with the further assumption of independence when $d$ ranges
over divisors of $p-1$, this suggests that 
$$
\#\cA(N) = (1+ o(1))  H(N)
$$
as $N\to\infty$, where 
\begin{equation}
\label{eq:HN}
H(N) = \sum_{p < N}
\prod_{\substack{d|p-1 \\ 2<d<p-1}} \(1-\frac{\fl{(p-3)/d} }{p-3}\)^{\varphi(d)}.
\end{equation}

For $p$ fixed (but large)  we  similarly find that the heuristic
probability of 
the map $\psi_p$ having no (nontrivial) fixed points, using that 
$$
\(1-\(\frac{1}{d}+O(p^{-1})\)\)^{\varphi(d)} = \exp
\(  \varphi(d) \ln\(1 - \(\frac{1}{d}+O(p^{-1})\)\)\).
$$
is given by 
$\exp(-\Delta_p)$, where 
\begin{equation}
\label{eq:Delta-p}
\begin{split}
\Delta_p&:=-\sum_{\substack{
d|p-1 \\ 2 <d < p-1           }}
\varphi(d)  \ln\(1 - \(\frac{1}{d}+O(p^{-1})\)\)\\
& =  \sum_{\substack{
d|p-1 \\ 2 <d < p-1}}
\varphi(d) \( \frac{1}{d} +\frac{1}{2d^{2}} + O(p^{-1}+d^{-3}) \) \\
&=  \sum_{d|p-1}
\varphi(d)  \( \frac{1}{d} +\frac{1}{2d^{2}} + O(p^{-1}+d^{-3}) \) +O(1)  \\
&=  \tau(p-1) \cdot \prod_{q^{e}||p-1}
\left(1-\frac{e}{(1+e)q} \right)   + O\(
\sum_{d|p-1} 1/d\).
\end{split}
\end{equation}

Hence, $\psi_p$ is exceeding likely to have a nontrivial fixed point
unless $p-1$ have rather few prime factors.
Restricting to $p$ such that $p-1$ is square-free, and, motivated by
the  results of
Sathe~\cite{Sathe}
and Selberg~\cite{Sel}, assuming that  for any fixed $\varepsilon>0$ 
and $k \leq (2-\varepsilon) \log_2 N$, 
we have
$$
\#\{p \le N~:~\omega(p-1) = k\}
\sim \frac{N (\log_2 N)^{k-1}}{(k-1)! \log^2 N}
$$
we  expect that the number of $p \leq x$ such that $\psi_p$ has
no nontrivial fixed point modulo $p$ is, for any integer $k>0$,
is 
$$
 H(N)  \gg 
\sum_{p \le N} \exp(-\Delta_p) \ge \sum_{1 \le k \le  (2-\varepsilon) \log_2 N}
 \frac{N (\log_2 N)^{k-1} \exp(-2^{k+o(k)})  }{(k-1)!  \log^2 N}.
$$

Using the trivial estimate $1 \le  (k-1)! \le k^k$ we see that  $(k-1)!$
can be absorbed in $2^{k+o(k)}$ in the exponent. Furthermore, for any positive 
integer $k \le  (2-\varepsilon) \log_2 N$ we have
$$
H(N) \gg \frac{N \exp\(k\log_3 N-2^{k+o(k)}\)}{(\log N)^2 \log_2 N}.
$$
Thus, taking
$$
k=\fl{\(\frac{1}{\ln 2} - \eta\)\log_4 N},
$$ 
for an arbitrary $\eta > 0$ gives the bound 
$$
H(N)\ge \frac{N}{(\log N)^2} \exp\((1/\ln 2- \eta +o(1))\log_3 N \log_4 N\)
$$
(note that using other admissible values of $k$ does not significantly
improve this bound; just one optimally chosen value suffices.)
Since $\eta > 0$ is arbitrary,  we obtain the expected 
lower bound~\eqref{eq:AN LB}.

In fact we believe that the lower bound~\eqref{eq:AN LB} is close to
the actual order of magnitude  of both  $\#\cA(N)$ and $H(N)$.

The above argument, in particular \eqref{eq:d-power-probability},
also suggests that the expected value 
of the total number of nontrivial fixed points over all primes $p \le N$
is 
$$
\sum_{p \le N} F(p) = (1+ o(1))  K(N)
$$
where 
\begin{equation}
\label{eq:KN}
K(N) = 
\sum_{p \le N} \sum_{\substack{d \mid p-1\\ d > 2}}\frac{\varphi(d)}{d}
=  \sum_{d = 3}^N\frac{\varphi(d)}{d} \sum_{\substack{p \le N \\ p \equiv 1 \pmod d}} 1 .
\end{equation}
Using the approximation
$$
 \sum_{\substack{p \le N \\ p \equiv 1 \pmod d}} 1 = (1+o(1)) \frac{N}{\varphi(d) \log N},
$$
it seems reasonable to expect that 
$$
K(N) = (1+ o(1)) N.
$$

\subsection{Numerical results}
\label{sec:numerics}

In Table~\ref{tab:F=1} we compare the observed data for all primes $p
\leq N$ for $N= 100000 \cdot k$, $1 \leq k \leq 10$, 
that have no nontrivial fixed point with the heuristically 
predicted value $H(N)$ given by~\eqref{eq:HN}.

\begin{table}[h]
  \centering
  \begin{tabular}{|r|r|r|r|}
\hline
$N$ & Observed & Predicted & Relative error\\
\hline
100000 & 567 & 585.6 & -0.0318\\
200000 & 1007 & 1020.6 & -0.0134\\
300000 & 1358 & 1421.4 & -0.0446 \\
400000 & 1715 & 1790.1 & -0.0419 \\
500000 & 2068 & 2151.8 & -0.0389 \\
600000 & 2404 & 2490.0 & -0.0345 \\
700000 & 2725 & 2826.7 & -0.0360 \\
800000 & 3053 & 3151.0 & -0.0311 \\
900000 & 3350 & 3479.5 & -0.0372 \\
1000000 & 3632 & 3796.2 & -0.0433 \\
\hline
\end{tabular}
\caption{Number of primes $p \leq N$ with no nontrivial fixed point}
  \label{tab:F=1}
\end{table}

In Table~\ref{tab:Av F} we present data for the total number of fixed
points for all primes $p \leq N$ for $N= 50000 \cdot k$, $1 \leq k
\leq 9$, that have no nontrivial fixed point, and compare it with with
the heuristically predicted value given by~\eqref{eq:KN}.

\begin{table}[h]
  \centering
  \begin{tabular}{|r|r|r|r|}
\hline
$N$ & Observed & Predicted & Relative error \\
\hline
500000 & 465413 & 410686.1 & 0.1333\\
1000000 & 936280 & 831872.7 & 0.1255\\
1500000 & 1408964 & 1256499.5 & 0.1213\\
2000000 & 1883411 & 1683081.9 & 0.1190\\
2500000 & 2357781 & 2110954.9 & 0.1169\\
3000000 & 2832933 & 2539862.9 & 0.1154\\
3500000 & 3306597 & 2968852.5 & 0.1138\\
4000000 & 3780495 & 3398836.9 & 0.1123\\
4500000 & 4256757 & 3829903.3 & 0.1115\\
\hline
\end{tabular}
\caption{Total number of observed nontrivial fixed points for $p \leq N$
  vs. random model prediction.} 
    \label{tab:Av F}
\end{table}

When comparing predicted and observed values we note that there seems
to be a consistent negative bias in 
Table~\ref{tab:F=1} and a consistent positive bias in
Table~\ref{tab:Av F}.  As
of now, we have no satisfactory explanation of this phenomenon.

\section{Remarks on the Dynamics of the Map $\psi_p$}

\subsection{Orbit length model}

Given a finite set $X$, a map $\eta:  \ X \to X$, and a starting point
$x_{0}$, define $x_{n+1} = \eta(x_n)$ for $n \in \Z^+$.  Let 
$O_{\eta,x_{0}}(X) = \{x_{0}, x_{1}, \ldots \}$ denote the forward orbit
of $x_{0}$ under $\eta$.  Clearly, we have 
the trivial inequality $\#O_{\eta,x_{0}}(X)  \leq \#X$, but if $\eta$ is a
{\em random map} (that is, for each $x\in X$, we define its image
  $\eta(x)$ by uniformly selecting a random element of $X$), a simple
`birthday paradox' argument shows that $\#O_{\eta,x_{0}}(X) $ is very
likely to be of 
size roughly $(\# X)^{1/2}$; in particular, as $\#X\to \infty$,
$\#O_{\eta,x_{0}}(X)  \leq 
(\#X)^{1/2+o(1)}$ holds with probability one.

Thus, if we na\"{\i}vely model $\psi_p$ as a random map, then,  as $p
\to \infty$, and selecting a random starting point $x_{0}$, the orbit size
$\# O_{\psi_p,x_{0}}(\F_p)$ is expected to be roughly of size $\sqrt{p}$,
see~\cite{FlOdl}.  
However, numerics indicate that $\# O_{\psi_p,x_{0}}(\F_p)$ often is
{\em much smaller} than $\sqrt{p}$.  In fact, in what follows, we  give
numerical evidence, and an heuristic model, that the probability
density distribution of $\log \# O_{\psi_p,x_{0}}(\F_p)/\log p$ has
support in $[0,1/2]$.

In fact, it is easy to see that  the orbit 
$ O_{\psi_p,x_{0}}(\F_p)$ are shorter than expected
from a random map as once  a certain element $x\in
O_{\psi_p,x_{0}}(\F_p)$ lies in a  
multiplicative subgroup of $\F_p^*$, then so does $\psi_p(x)$, and the
remaining part of  
the orbit never leaves this subgroup. 
 So, the behavior of orbits of $\psi_p$, originating at a point  $x_{0} \in \F_p^*$  is ruled by two (apparently independent)
factors:
\begin{itemize}
\item random map-like behavior inside of a subgroup of $\F_p^*$ which 
eventually leads to a  cycle formed by the `birthday paradox' (see~\cite{FlOdl}
for an exhaustive treatise of the structure of random maps);

\item reducing the size of the multiplicative  subgroup where the iterations of $\psi_p$ 
 get locked in as they progress along the trajectory. 
 \end{itemize}
 
For example, if the initial point $x_0$ is not a primitive root of $\F_p$, 
 this immediately puts all elements of the corresponding trajectory in 
 a nontrivial multiplicative  subgroup of $\F_p^*$.

Hence, we believe that the main reason for such small orbit lengths is that a
correct model for $\psi_{p}$ is that of a {\em random automorphism} on
$C_{p-1}$, the cyclic group of cardinality $p-1$.  Since $\psi_p$ maps
$\fp^{\times}$ into itself, and, as groups $\fp^{\times} \simeq
C_{p-1}$, we may translate the dynamics $x_{0}\to x_{1}\to \cdots$ on
$\fp^{\times}$ to dynamics $y_{0} \to y_{1} \to \cdots $ on $C_{p-1}$.
Under the assumption that the discrete log map (which identifies
$\fp^{\times}$ with $C_{p-1}$) behaves randomly, the image
 of $\psi_p$ as a map of $C_{p-1}$ be viewed as
``random'' map $\varphi:\ C_{p-1}\to C_{p-1}$ given by 
$$
\varphi(y) \equiv  \alpha_{y}  y \pmod{p-1},
$$ where $\alpha_y \in \Z/(p-1)\Z$ is selected
randomly.  In particular, once an iterate $y_{n}$ ``lands'' in a
subgroup $H \subset C_{p-1}$, it never ``leaves''; and this makes 
much shorter orbit lengths likely.  

For example, for primes $p$ such that $p-1 = s \cdot t$, where $s$ is the
$p^{1/3}$-smooth part of $p-1$, and $s \gg p^{1/3}$, we find that it
is very likely that the $s$-part of the orbit gets annihilated after at
most $p^{1/3+\varepsilon}$ steps (write $C_{p-1} \simeq C_{s} \times
C_{t}$ and say that the $s$-part of $y_{n}$ is annihilated if the image of
$y_{n}$ in $C_{s} \times C_{t}$ is of the form $(0,*)$.)  In fact, if
a prime $q$ divides $p-1$, it is easy to see that the probability of
the $q$-part not being annihilated after $k$ steps is given by
$(1-1/q)^{k}$, which, if $ q/k = o(1)$, is $o(1)$ as $q\to\infty$.

This leads to the following
natural question. Let $\Psi_{d,p}$ be the endomorphisms 
of $\F_p^*$, (indexed by the divisors $d \in [1, p-1]$) and generated
by the map $x\mapsto x^d$, $x \in \F_p^*$.

\begin{ques}
 \label{ques:Endom} 
Let  $x_{0} \in \F_p^*$ be chosen uniformly at random and let 
 $\Psi_{d_1,p}, \ldots, \Psi_{d_L,p}$  be a 
sequence of $L$ random endomorphisms such 
that for every $j=1, \ldots, L$ and $d\mid p-1$ we have 
$$
\Pr[(p-1,d_j) = d] = \frac{\varphi((p-1)/d)}{p-1},
$$
What is the expected size of the smallest subgroup of $\F_p^*$ 
that contains the element $\Psi_{d_L,p}\( \ldots \(\Psi_{d_L,p}(x_0)\)\ldots\)$? 
\end{ques}

Certainly, a version of Question~\ref{ques:Endom} can be asked for any
finite subgroup.

\subsection{Orbit length statistics}

If $\eta$ behaves sufficiently randomly, then 
$\# O_{\eta,x_{0}}(\F_p) \le p^{1/2+o(1)}$
is very likely to hold.  In fact, it is known that for $\eta$ random,
$\(\# O_{\eta,x_{0}}(\F_p)\)^{2}/(2p)$ converges in distribution to a mean one exponential
as $p \to \infty$.  In particular,  the support of
$\log \# O_{\eta,x_{0}}(\F_p)/\log p$ is essentially concentrated around $1/2$.

See Figure~\ref{fig:fig12} for an illustration of this 
well-known phenomenon, which also forms the basis of the so-called 
{\it Pollard's rho-factorisation\/} algorithm, see~\cite[Section~5.2.1]{CrPom}.

\begin{figure}[H]
  \centering
\includegraphics[width=6cm]{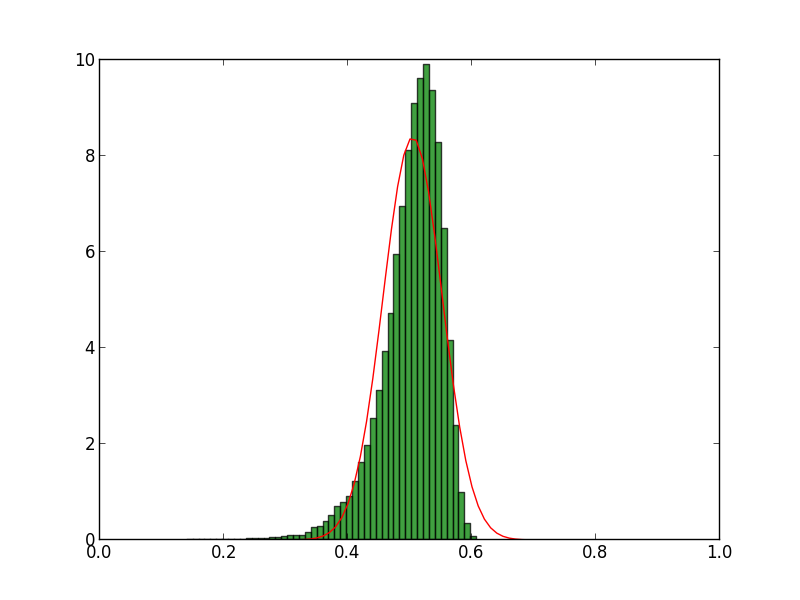}  
\includegraphics[width=6cm]{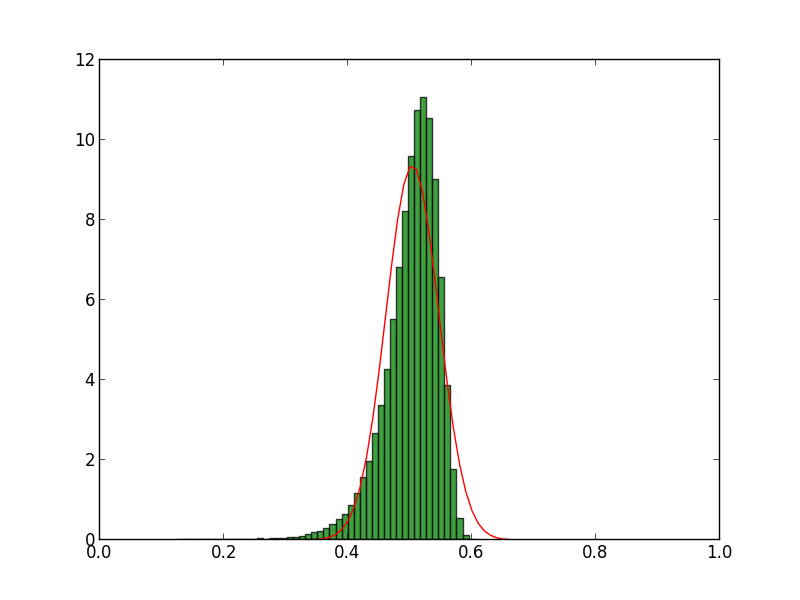}  
\caption{Histogram plot of $\log\# O_{\eta,x_{0}}(\F_p)/\log p$ with
  $\eta(x) = x^{2}+1$ for $p \leq 1000000$ (left) and $p \leq 5000000$
  (right).  Red curves indicate normal distributions with mean and
  variance fitted to the data.}
  \label{fig:fig12}
\end{figure}

However, the orbit sizes of $\psi_p$ behaves very differently.

We remark that if $p= 2q +1$ where $q$ is a Sophie Germain prime, then
the second effect is negligible.  Since the standard heuristic
suggests a (relative) abundance of Sophie Germain primes, ``on
average'' over primes $p$, the second effect is essentially
invisible. However for a ``typical'' prime the situation is quite
different.  In other words, under the standard heuristic expectation
of abundance of Sophie Germain prime, the average value of the
trajectory length is of order $p^{1/2}$ (possibly with some
logarithmic factors), while the typical value is much smaller.

Furthermore, let $P(k)$ denote the largest prime divisor of an integer $k\ge 1$.
If $\alpha \in (0,1)$ and $p$ runs through a sequence of primes with
$p-1 = q \cdot s$ where $q = P(p-1) = p^{\alpha+o(1)}$ and $s$ is
$p^{\alpha/2}$-smooth (which conjecturely holds for a positive
proportion of the primes for any $\alpha \in (0,1)$), we expect that a
random endomorphism has the orbit of size at most $p^{\alpha/2+o(1)}$.
In turn, this suggests that the probability density function of
$\log\# O_{\psi_p,x_{0}}(\F_p)/\log p$ is supported in the full
interval $[0,1/2]$;
see Figure~\ref{fig:fig34} for an illustration of
this phenomenon.

\begin{figure}[H]
  \centering
\includegraphics[width=6cm]{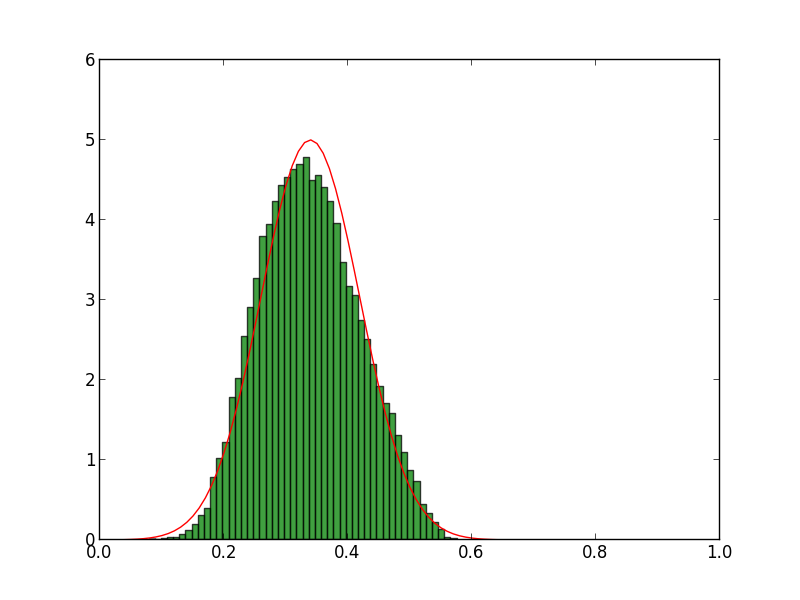}  
\includegraphics[width=6cm]{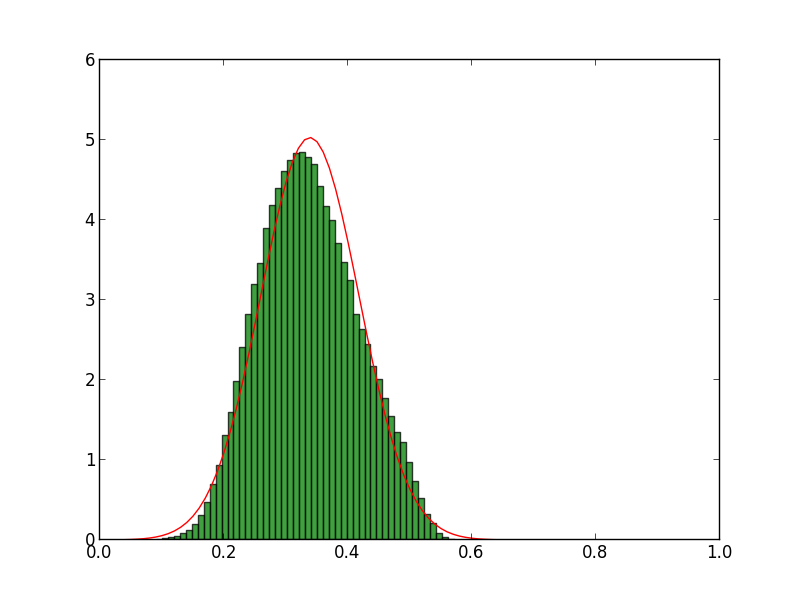}  
\caption{Histogram plots of $\log \# O_{\psi_p,x_{0}}(\F_p)/\log
  p$, $p \leq 1000000$ (left) and $p \leq 5000000$ (right).
 Red curves indicate normal distributions
with mean and variance fitted to the data.} 
\label{fig:fig34}
\end{figure}

To further show the difference in orbit statitics, it is also
interesting to compare statisticics when normalized by dividing by
$\sqrt{p}$, see Figure~\ref{fig:fig56}. 

\begin{figure}[H]
  \centering
\includegraphics[width=6cm]{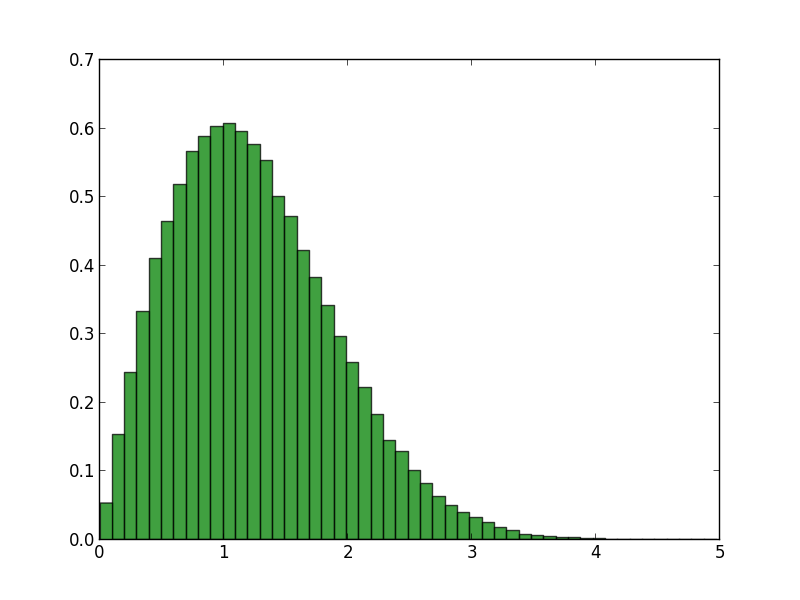}  
\includegraphics[width=6cm]{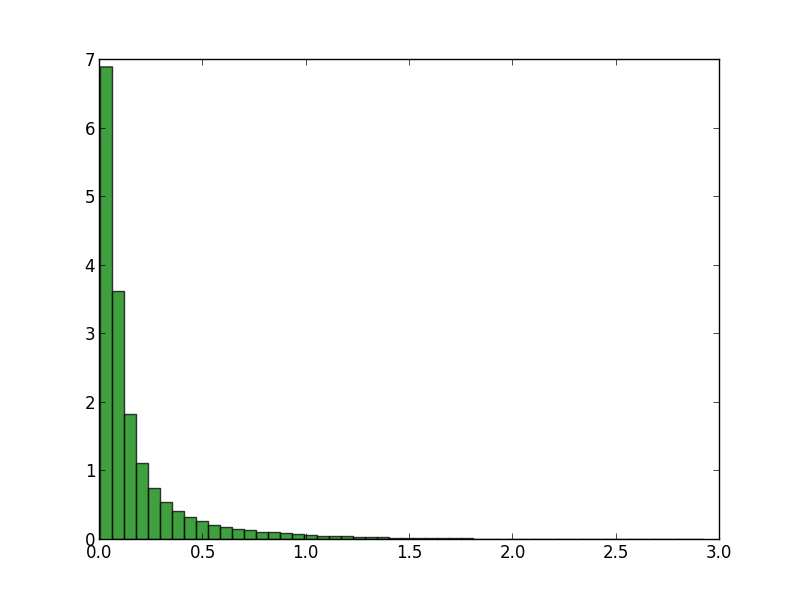}  
\caption{Histogram plots of $\# O_{\eta,x_{0}}(\F_p)/\sqrt{p}$ with
  $\eta(x) = x^{2}+1$ (left) and $\# O_{\psi_p,x_{0}}(\F_p)/\sqrt{p}$
  (right) for $p \leq 5000000$.}
\label{fig:fig56}
\end{figure}
 
\section{Comments and Extensions}
\label{sec:comm ext}

As we have mentioned in Section~\ref{sec:syst-lin-form}, it is natural 
to expect that the following holds:

 \begin{conj}
 \label{conj:LinForm} 
 Let $x_{0} \in \fq$.  Then
$$
\frac{\#\{ \vec{v} \in \fq^{d}~:~L_n(\vec{v})  \neq x_{0}~\text{for $1 \leq n
    \leq q$} \}| }{q^{d}} = e^{-1} + o(1), 
$$
as $q\to\infty$, where $d = \pi(q-1)$. 
\end{conj}

In particular, Conjecture~\ref{conj:LinForm} implies that $\vartheta \simeq 0.4231\ldots$
in the bound of Theorem~\ref{thm:main} can be replaced with $1-1/e \simeq 0.6321\ldots$.

Clearly the map $\psi_p$, as any map over $\F_p$ can be interpolated by polynomial, 
that is, for some unique polynomials $F_p(X) \in \F_p[X]$ of degree at most $p-1$  we have 
$\psi_p(x) = F_p(x)$ for $x \in \F_p$. 
It is natural to use   $D_p = \deg F_p$ as a measure of 
``non-polynomiality'' of the map $\psi_p$. 
In particular, we expect that $D_p$ is 
close to its largest possible value $p-1$. Although we have 
not been able to establish this we  show that 
\begin{equation}
\label{eq:Dp}
D_p \ge \(\sqrt{2-\sqrt{3}}+o(1)\) p^{1/2} = 0.5176 \ldots p^{1/2}.
\end{equation} 
We remark that the $x^x$ is a quadratic non-residue modulo $p$
if and only if both $x$ is odd and a quadratic non-residue.
Using the P{\'o}lya--Vinogradov bound of sums of quadratic characters,
it is trivial to show that there are $p/4 + O(p^{1/2}\log p)$ 
such values of $x = 0,1\ldots, p-1$. 
Hence, for the sum of the Legendre symbols with $F_p$ we have 
$$
\sum_{x \in \F_p} \(\frac{F_p(x)}{p}\) = p/2  + O(p^{1/2}\log p).
$$
On the other hand, the results of Korobov~\cite{Kor} and 
Mit'kin~\cite{Mit} (which we use in a simplified form) imply that 
$$
\left|\sum_{x \in \F_p} \(\frac{F_p(x)}{p}\) \right|\le 
D_p\sqrt{p - D_p^2/4 + O(D_p)}
$$
(provided that, say, $p\ge D_p^2/2 +5$), 
which now implies~\eqref{eq:Dp}. 

For a prime $p$ and a polynomial $f(X) \in \Z[X]$ we denote by 
$T_{f}(p)$ 
the number of solutions to the congruence
\begin{equation}
\label{eq:Cong}
x^{f(x)} \equiv 1 \pmod p,  \qquad  1 \le x \le p-1.
\end{equation} 
We note that the number of fixed points of  $x \to x^{f(x)}$ is
given by $T_{f-1}(p)$.

\begin{theorem}
\label{thm:Tfp} 
If $f$ is squarefree, we have 
$$
T_{f}(p)  \le  p^{6/13 + o(1)}
$$ 
as $p\to \infty$.
\end{theorem}

\begin{proof}
Let us fix $d \mid p-1$ and denote by $\cX_d$
the set of solutions to~\eqref{eq:Cong} with 
$$
\gcd(f(x),p-1) = d.
$$ 
Clearly any element $x \in \cX_d$ belongs to the multiplicative
group $\cG_d \subseteq \F_p^*$ of index $d$ in the multiplicative
group $\F_p^*$ of a finite field $\F_p$ of $p$ elements.
Therefore, 
\begin{equation}
\label{eq:Bound 1}
\# \cX_d \le d.
\end{equation}

Since $f(X)$ is squarefree, by the Nagell--Ore theorem
(see~\cite{Huxley} for its strongest known form) for each $d$ there is
a set $\cK_d \subseteq\{0, \ldots, d-1\}$ of cardinality 
$\# \cK_d = d^{o(1)}$ and such that 
every  $x \in \cX_d$  satisfies 
\begin{equation}
\label{eq:Cong k}
x \equiv k \pmod d
\end{equation}
for some $k \in \cK_d$.
Let us fix $k \in \cK_d$ and denote by $\cX_{d,k}$ the set 
of $x \in \cX_d$ satisfying~\eqref{eq:Cong k}. 
Obviously, 
\begin{equation}
\label{eq:Bound 2}
\# \cX_{d,k} \le (p-1)/d.
\end{equation}
Thus, in particular, from~\eqref{eq:Bound 1} and~\eqref{eq:Bound 2},
we see that  $\# \cX_{d,k} \le \sqrt{p-1}$. However, we now obtain a 
better bound. 

We remark that the difference set
$$
\cU_{d,k}  =\{x_1 - x_2 \ : \ x_1, x_2 \in \cX_{d,k}\} \subseteq \F_p
$$
is of cardinality at most
\begin{equation}
\label{eq:U}
\# \cU_{d,k} \le  2(p-1)/d
\end{equation}
as it is  contained in the reductions modulo $p$ of 
integers $y \equiv 0 \pmod d$ from the interval $y \in [-(p-1), p-1]$.
Similarly, for 
$$
\cW_{d,k} =\{x_1 +x_2- x_3- x_4 \ : \ x_1, x_2, x_3, x_4  \in \cX_{d,k}\} \subseteq \F_p,
$$
we have
\begin{equation}
\label{eq:V}
\# \cV_{d,k}  \le  4(p-1)/d.
\end{equation}
Furthermore, the product set 
$$
\cW_{d,k}  =\{x_1  x_2 \ : \ x_1, x_2 \in \cX_{d,k}\} \subseteq \F_p
$$
is of cardinality at most
\begin{equation}
\label{eq:W}
\# \cW_{d,k}   \le d
\end{equation}
as it is  contained in $\cG_d$. 
Finally, as in~\cite[Section~1]{BouGar}, we note that the Cauchy 
inequality implies that
$$
E_{d,k} = \# \{(x_1, x_2, x_3, x_4)  \in \cX_{d,k}^4 \ : \
x_1x_2 = x_3x_4\}
$$
satisfies
\begin{equation}
\label{eq:Energy}
E_{d,k}  \ge \frac{(\#\cX_{d,k} )^4} {\#  \cW_{d,k}}.
\end{equation}

By the result  of Bourgain and Garaev~\cite[Theorem~1.1]{BouGar} we
have
$$
E_{d,k}^4  \le \(\# \cU_{d,k} +  \frac{(\#\cX_{d,k} )^3} {p}\)
(\#\cX_{d,k} )^5  \(\#\cU_{d,k}\)^4\# \cV_{d,k} p^{o(1)}, 
$$
which together with~\eqref{eq:Energy} implies
\begin{equation}
\label{eq:BG}
(\#\cX_{d,k} )^{11}  \le \(\# \cU_{d,k} +  \frac{(\#\cX_{d,k} )^3} {p}\)
 \(\#\cU_{d,k}\)^4\# \cV_{d,k}  \(\#\cW_{d,k}\)^4 p^{o(1)}
\end{equation}
as $p\to\infty$.
Substituting the bounds~\eqref{eq:U}, \eqref{eq:V} and~\eqref{eq:W}
in~\eqref{eq:BG}, we derive 
$$
(\#\cX_{d,k} )^{11}  \le \(pd^{-1} +  \frac{(\#\cX_{d,k} )^3} {p}\)
 p^{5+o(1)} d^{-1}. 
$$
Thus, 
\begin{equation}
\label{eq:Bound 3}
\#\cX_{d,k} \le \max\left\{p^{6/11} d^{-2/11}, p^{1/2} d^{-1/8}\right\} p^{o(1)}.
\end{equation}
Using~\eqref{eq:Bound 1} for $d < p^{6/13}$ and~\eqref{eq:Bound 1}
for  $d \ge p^{6/13}$, we obtain 
$$
\#\cX_{d,k} \le p^{6/13+o(1)}, 
$$
as $p\to\infty)$, which concludes the proof.
\end{proof}

\begin{remark}
We note that as long as $d$ is square free, we have $\# \cK_d =
d^{o(1)}$ with no assumption of  $f$ being square free. Hence, we find that the upper bound on $T_{f}(p)$ holds without any
assumption on $f(x)$ provided that $p-1$ is square free.  In fact, it
is enough to assume that the square full part of $p-1$ is of size
$p^{o(1)}$.
\end{remark}

\begin{remark}
It is quite possible that using the results and arguments 
of Rudnev~\cite{Rud} one can improve the bound of Theorem~\ref{thm:Tfp}. 
\end{remark}

\section*{Acknowledgements}

Part of this work was done during visits of
F.~L. at KTH, Stockholm and Macquarie University, Australia and
P.~K. at the Mathematical Institute of the UNAM in Morelia, Mexico.
These authors thank these institutions for their hospitality and
support. 

P.~K. was partially supported by grants from the G\"oran
Gustafsson Foundation, the Knut and Alice Wallenberg foundation, the
Royal Swedish Academy of Sciences, and the Swedish Research Council,
F.~L.  was supported in part by Grants PAPIIT 104512, CONACyT 163787, CONACyT 193539  and a Marcos Moshinsky Fellowship, and I.~E.~S.\ was supported in part by ARC Grant
DP1092835.

\end{document}